\documentclass[12pt]{article}
\usepackage{ct}

\usepackage[utf8]{inputenc}
\usepackage{mathrsfs}
\usepackage{caption,cite}
\usepackage{mathtools,color}

\newtheorem{thm}{Theorem}[section]
\newtheorem{cor}[thm]{Corollary}

\newtheorem{lem}[thm]{Lemma}
\newtheorem{conj}[thm]{Conjecture}
\newtheorem{prop}[thm]{Proposition}
\newtheorem{quest}[thm]{Question}
\newtheorem{prob}[thm]{Problem}

\theoremstyle{definition}
\newtheorem{defn}{Definition}

\newcommand{\R}{\mathbb{R}}

\newcommand{\E}{\mathbb{E}}

\newcommand{\Ex}{\mathrm{E}}

\renewcommand{\l}{\left}
\renewcommand{\r}{\right}
\newcommand{\ex}{\mathrm{ex}}

\renewcommand{\c}[1]{\mathcal{#1}}
\newcommand{\sub}{\subseteq}
\newcommand{\Om}{\Omega}
\newcommand{\sm}{\setminus}
\newcommand{\f}[2]{\frac{#1}{#2}}

\newcommand{\al}{\alpha}
\newcommand{\ep}{\epsilon}

\newcommand{\del}{\delta}
\newcommand{\half}{\frac{1}{2}}

\newcommand{\Hom}{\mathrm{Hom}}

\specs{n}{vol (issue)}{year}{}

\dateline{Nov 15, 2023}{Jun 16, 2025}{TBD}



\title{Sidorenko Hypergraphs and Random Tur\'an Numbers}

\author[1]{Jiaxi Nie}
\author[2]{Sam Spiro\thanks{This material is based upon work supported by the National Science Foundation Mathematical Sciences Postdoctoral Research Fellowship under Grant No. DMS-2202730.}}

\affil[1]{%
School of Mathematics, Georgia Institute of Technology, Atlanta, GA, U.S.A.

\email{jnie47@gatech.edu}%
}

\affil[2]{%
Department of Mathematics, Rutgers University, Piscataway, NJ, U.S.A.

\email{sas703@scarletmail.rutgers.edu}%
}

\begin{document}

\maketitle


\begin{abstract}
Let $\mathrm{ex}(G_{n,p}^r,F)$ denote the maximum number of edges in an $F$-free subgraph of the random $r$-uniform hypergraph $G_{n,p}^r$, and let $s(F):=\sup\{s: \exists H,\ t_F(H)=t_{K_r^r}(H)^{s+e(F)}>0\}$.  Following recent work of Conlon, Lee, and Sidorenko, we prove non-trivial lower bounds on $\mathrm{ex}(G_{n,p}^r,F)$ whenever $s(F)>0$, i.e. $F$ is not Sidorenko. This connection between Sidorenko's conjecture and random Tur\'an problems gives new lower bounds on $\mathrm{ex}(G_{n,p}^r,F)$ whenever $s(F)>0$, and further allows us to establish upper bounds for $s(F)$ whenever upper bounds for $\mathrm{ex}(G_{n,p}^r,F)$ are known. As a consequence, we prove that $s(\Ex^r(K_{k+1}^k))=\frac{1}{r-k}$ where $\Ex^r(K_{k+1}^k)$ is the $r$-expansion of $K_{k+1}^k$.
\end{abstract}



\section{Introduction}

In recent work, Conlon, Lee, and Sidorenko~\cite{conlon2023extremal} established a connection between two seemingly unrelated problems in extremal combinatorics: Sidorenko's conjecture and Tur\'an problems. Building on their work, we further establish a connection between Sidorenko's conjecture and \textit{random} Tur\'an problems. Along the way, we obtain good estimations for how ``far'' certain hypergraphs $F$ are from being Sidorenko.

Throughout this paper we consider $r$-uniform hypergraphs, or $r$-graphs for short. For an $r$-graph $H=(V,E)$, we use the notation $v(H):=|V|$ and $e(H):=|E|$.

\subsection{Sidorenko's Conjecture}
Recall that a \textit{homomorphism} from an $r$-graph $F$ to an $r$-graph $H$ is a map $\phi:V(F)\to V(H)$ such that $\phi(e)$ is an edge of $H$ whenever $e$ is an edge of $F$.  We let $\hom(F,H)$ denote the number of homomrophisms from $F$ to $H$ and define the \textit{homomorphism density}
\[t_F(H)=\frac{\hom(F,H)}{v(H)^{v(F)}},\]
which is equivalently the probability that a random map $\phi:V(F)\to V(H)$ is a homomorphism. Sidorenko famously conjectured the following.
\begin{conj}[Sidorenko's Conjecture \cite{Sidorenko1991Inequalities,Sidorenko1993Acorrelation}]
    For every bipartite graph $F$ and every graph $H$,
    \begin{equation}\label{equation:Sidorenko}
        t_F(H)\ge t_{K_2}(H)^{e(F)}.
    \end{equation}
\end{conj}
This bound is asymptotically best possible by considering random graphs. A large body of literature is dedicated towards Sidorenko's conjecture~\cite{conlon2010approximate,conlon2017finite,conlon2018sidorenko,conlon2018some,coregliano2021biregularity,fox2017local,hatami2010graph,kim2016two,li2011logarithimic,lovasz2011subgraph,szegedy2014information}, but overall this problem remains very wide open. We call a graph $F$ \emph{Sidorenko} if it satisfies \eqref{equation:Sidorenko}.  This notion naturally extends to hypergraphs as follows.
\begin{defn}
    We say that an $r$-graph $F$ is \textit{Sidorenko} if for every $r$-graph $H$ we have
    \[t_F(H)\ge t_{K_r^r}(H)^{e(F)},\]
    where $K_r^r$ is the $r$-graph consisting of a single edge.
\end{defn}
This bound is also asymptotically best possible by considering random hypergraphs.  Observe that if $F$ is not $r$-partite, then $F$ is not Sidorenko (due to $H=K_r^r$, for example).  As such, it suffices to only consider $r$-partite $r$-graphs when discussing Sidorenko hypergraphs. 

It is known that Sidorenko's conjecture does not extend to hypergraphs.  In particular, Sidorenko~\cite{Sidorenko1993Acorrelation} showed that $r$-uniform loose triangles (see \Cref{def:expand}) are not Sidorenko for $r\ge 3$.  Because of this, there seems to have been relatively little interest in studying which $r$-graphs are Sidorenko.  This changed very recently with the work of Conlon, Lee, and Sidorenko \cite{conlon2023extremal} who showed the following surprising connection between non-Sidorenko hypergraphs and Tur\'an numbers.  Here we recall that the \textit{Tur\'an number} $\ex(n,F)$ of an $r$-graph $F$ is the maximum number of edges that an $n$-vertex $F$-free $r$-graph can have.

\begin{thm}[\cite{conlon2023extremal}]\label{thm:CLS}
    If $F$ is not Sidorenko, then there exists $c=c(F)$ such that
    \[\ex(n,F)=\Om\left(n^{r-\frac{v(F)-r}{e(F)-1}+c}\right).\]
\end{thm}
We note that this improves upon the trivial lower bound $\ex(n,F)=\Om(n^{r-\frac{v(F)-r}{e(F)-1}})$ which holds for all $F$ by a standard random deletion argument.  A number of other results around non-Sidorenko hypergraphs were proven in \cite{conlon2023extremal}, such as the fact that $r$-partite $r$-graphs of odd girth are always non-Sidorenko.

\subsection{The Random Tur\'an Problem}

We now discuss our second problem of interest:  the random Tur\'an problem. Given $r$-graphs $G,F$, we define $\ex(G,F)$ to be the maximum number of edges in an $F$-free subgraph of $G$. Define $G_{n,p}^r$ to be the random $r$-graph on $n$ vertices obtained by including each possible edge independently and with probability $p$.  When $r=2$ we simply write $G_{n,p}$ instead of $G_{n,p}^2$.  We call $\ex(G_{n,p}^r,F)$ the \textit{random Tur\'an number}, which is the maximum number of edges in an $F$-free subgraph of $G_{n,p}^r$.  Note that when $p=1$ we have $\ex(G_{n,1}^r,F)=\ex(n,F)$, so this can be viewed as a probabilistic analog of the classical Tur\'an number.

The asymptotics of $\ex(G_{n,p}^r,F)$ has essentially been determined whenever $F$ is not an $r$-partite $r$-graph due to independent breakthrough work of Schacht \cite{schacht2016extremal} and of Conlon and Gowers \cite{conlon2016combinatorial}, and as such we will focus only on the case when $F$ is $r$-partite.   For $r=2$, McKinely and Spiro \cite{mckinley2023random} recently conjectured the following, where here we define the $r$-density of an $r$-graph $F$ with $v(F)>r$ by
\[m_r(F):=\max_{F'\subseteq F:v(F')>r}\l\{\frac{e(F')-1}{v(F')-r}\r\},\]
and we say that $F$ is \textit{$r$-balanced} if $m_r(F)=\frac{e(F)-1}{v(F)-r}$.
\begin{conj}[\cite{mckinley2023random}]\label{conj:MS}
    If $F$ is a graph with $\ex(n,F)=\Theta(n^\al)$ for some $\al\in (1,2]$, then a.a.s.\
    \[\ex(G_{n,p},F)=  \begin{cases}\Theta(p^{\al-1}n^\al) & p\ge n^{\f{2-\al-1/m_2(F)}{\al-1}} (\log n)^{O(1)}\\n^{2-1/m_2(F)}(\log n)^{O(1)} & n^{\f{2-\al-1/m_2(F)}{\al-1}}(\log n)^{O(1)} \ge p\ge n^{-1/m_2(F)},\\ (1+o(1))p {n\choose 2} & n^{-1/m_2(F)}\gg p\gg n^{-2}.\end{cases}\]
\end{conj}

In particular, \Cref{conj:MS} predicts that for bipartite graphs, $\ex(G_{n,p},F)$ always has a ``flat middle range'' where $\ex(G_{n,p},F)=n^{2-1/m_2(F)}(\log n)^{O(1)}$ for the entire range $n^{-1/m_2(F)}\le p\le n^{\frac{2-\al-1/m_2(F)}{\al-1}}(\log n)^{O(1)}$; see \Cref{fig:enter-label} for an example.  \Cref{conj:MS} is known to hold (assuming certain conjectures regarding $\ex(n,F)$) for complete bipartite graphs and even cycles \cite{morris2016number}, and for theta graphs \cite{mckinley2023random}, with the lower bound being known to hold for powers of rooted trees \cite{spiro2024random}.

\begin{figure}
    \centering
    \includegraphics[width=0.4\textwidth]{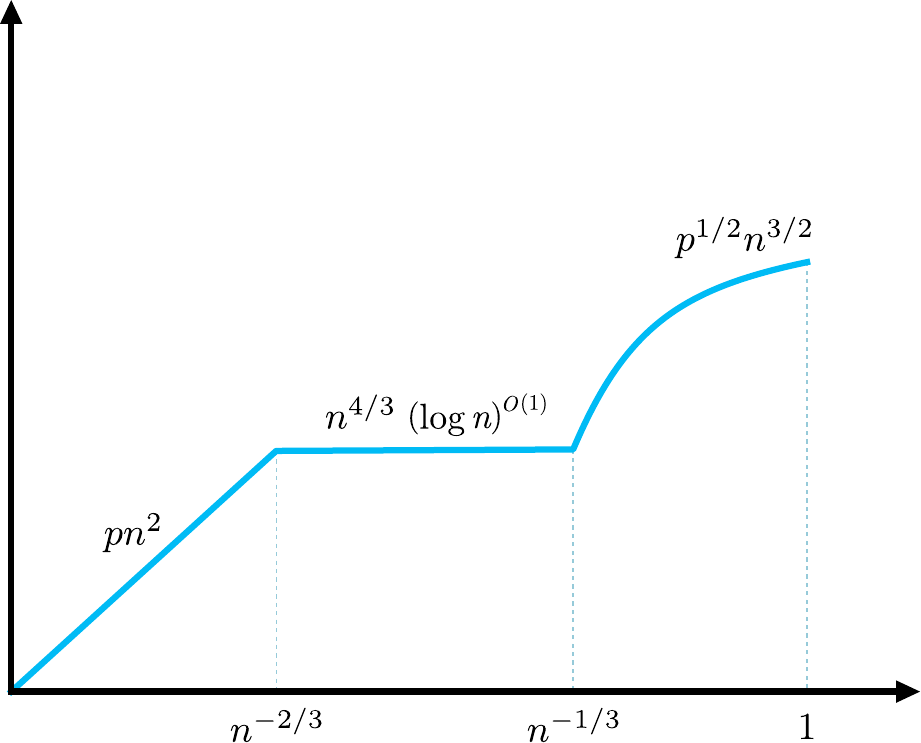}
    \caption{The plot of $\ex(G_{n,p},C_4)$ as proven by F\"uredi~\cite{F}.  More generally, \Cref{conj:MS} predicts that $\ex(G_{n,p},F)$ should have a ``flat middle range'' for every bipartite graph $F$ starting at $p=n^{-1/m_2(F)}$.}
    \label{fig:enter-label}
\end{figure}

It is known that \Cref{conj:MS} does not extend to hypergraphs.   In particular, Nie, Spiro, and Verstra\"ete \cite{nie2021triangle} and Nie \cite{nie2023random} showed this does not hold for $r$-uniform loose triangles with $r\ge 3$.

\subsection{Our Results}

The reader may notice that the situation for Sidorenko $r$-graphs and for random Tur\'an problems parallel each other quite closely, in particular with regard to loose triangles serving as a counterexample to each problem.  The main result of this paper (\Cref{thm:main}) shows that this is not a coincidence: in many cases, $r$-graphs which are not Sidorenko have a stronger lower bound on $\ex(G_{n,p}^r,F)$ than the generalization of \Cref{conj:MS} would predict. 

We prove this by extending \Cref{thm:CLS} to give lower bounds for random Tur\'an numbers.  In addition to this, we give a quantitative statement which in several cases gives optimal bounds for $\ex(G_{n,p}^r,F)$.  To state the quantitative result, we define\footnote{We emphasize that we define $s(F)$ in terms of a supremum rather than a maximum.  As such, we do not require there to exist an $H$ satisfying $t_F(H)=t_{K_r^r}(H)^{s+e(F)}$, though we note that such an $H$ can be guaranteed if one shifts to the essentially equivalent perspective of hypergraphons.} for an $r$-graph $F$ the quantity 
\[s(F):=\sup\{s: \exists H,\ t_F(H)=t_{K_r^r}(H)^{s+e(F)}>0\}.\]
Note that $s(F)=0$ if and only if $F$ is Sidorenko, and more generally, $s(F)$ measures how ``far'' an $r$-graph is from being Sidorenko.  

The following is our main result.

\begin{thm}\label{thm:main}
    If $F$ is an $r$-graph with $e(F)\ge 2$ and $\frac{v(F)-r}{e(F)-1}<r$,  then for any $p=p(n)\ge n^{-\frac{v(F)-r}{e(F)-1}}$, we have a.a.s.
    \[\ex(G_{n,p}^r,F)\ge n^{r-\frac{v(F)-r}{e(F)-1}-o(1)}(p n^{\frac{v(F)-r}{e(F)-1}})^{\frac{s(F)}{e(F)-1+s(F)}}.\]
\end{thm}
Our strategy for proving \Cref{thm:main} is based off \cite{conlon2023extremal} with a more careful execution. Because of this, the bound of \Cref{thm:main} will turn out to always be at least as strong as the (implicit quantitative version of) \Cref{thm:CLS}; see \Cref{cor:classicTuran} and the surrounding discussion for more.  In particular, the quantitative bounds of \Cref{thm:main} combined with known results for random Tur\'an numbers can be used to give effective (or even tight) bounds on $s(F)$ for a number of hypergraphs; see Corollaries~\ref{cor:upperBounds} and \ref{cor:Simplex} for more.


We emphasize that Theorem~\ref{thm:main} is only non-trivial when $s(F)>0$, i.e.\ when $F$ is not Sidorenko.  When $s(F)>0$ and $F$ is $r$-balanced, \Cref{thm:main} shows that $\ex(G_{n,p}^r,F)$ does not have a ``flat middle range'' as predicted in the case of graphs by  \Cref{conj:MS}.  

Assuming $F$ has no isolated vertices, then the only $F$ not satisfying our conditions are matchings, which are known to be Sidorenko. Furthermore, the random Turán problem for matching of size two (i.e. the EKR problem in random hypergraph) has been essentially solved by Balogh, Bohman, and Mubayi~\cite{balogh2009erdHos}, see also~\cite{hamm2019erdHosI,hamm2019erdHos,gauy2017erdHos,balogh2023sharp}. Their results have recently been extended to matchings of any size by Frankl, Nie, and Wang~\cite{frankl2024matching}.

To get effective bounds in \Cref{thm:main}, one must determine how large $s(F)$ is, i.e.\ how ``far'' from Sidorenko $F$ is; and we believe this to be a problem of independent interest.
\begin{prob}\label{prob:sF}
    Given an $r$-graph $F$, determine (bounds for) $s(F)$.
\end{prob}

\Cref{thm:main} allows us to translate  between the random Tur\'an problem and the Sidorenko-type problem \Cref{prob:sF}.  Specifically, \Cref{thm:main} shows that lower bounds on $s(F)$ give lower bounds for the random Tur\'an problem, and conversely, upper bounds on the random Tur\'an problem give upper bounds on $s(F)$.  For example, the non-Sidorenko $r$-graphs established by Conlon, Lee, and Sidorenko~\cite[Theorem 3.1]{conlon2023extremal} gives the following result for the random Tur\'an problem. 
\begin{cor}
    Let $F$ be an $r$-partite $r$-graph of odd girth.  Then there exists an $\ep>0$ such that if $p\ge n^{-\frac{v(F)-r}{e(F)-1}}$, then 
    a.a.s.
    \[\ex(G_{n,p}^r,F)\ge n^{r-\frac{v(F)-r}{e(F)-1}-o(1)}(p n^{\frac{v(F)-r}{e(F)-1}})^{\ep}.\]
\end{cor}

On the other hand, the simplest case of $p=1$ in \Cref{thm:main} gives the following general bound on $s(F)$.

\begin{cor}\label{cor:General}
    If $F$ is an $r$-graph with $\ex(n,F)=O(n^\al)$ and $\al<r$, then
    \[s(F)\le \frac{v(F)-\al}{r-\al}-e(F).\]
\end{cor}

While \Cref{cor:General} can sometimes provide tight bounds (see \Cref{cor:Simplex}), in general more effective bounds are obtained by considering smaller values of $p$.  To state these improved bounds, we make the following definitions which will be used throughout the paper.
\begin{defn}\label{def:expand}
    Given a $k$-graph $F$, we define the $r$-expansion $\Ex^r(F)$ to be the $r$-graph obtained by enlarging each $k$-edge of $F$ with a set of $r-k$ vertices of degree one.  When $F$ is the graph cycle $C_\ell$, we write $C_{\ell}^r:=\Ex^r(C_\ell)$ and refer to this as the $r$-uniform \textit{loose cycle} of length $\ell$.
\end{defn}
Lastly, we define an $r$-graph $T$ to be a {\em tight $r$-tree} if its edges can be ordered as $e_1,\dots,~e_t$ so that 
$$
\forall i\ge 2~\exists v\in e_i~and~1\le s\le i-1~such~that~v\not\in\cup_{j=1}^{i-1}e_j~and~e_i-v\subset e_s.
$$
The following bounds for $s(F)$ turn out to follow from \Cref{thm:main} together with known bounds for random Tur\'an numbers \cite{mubayi2023random,nie2023random,nie2024turan}; see the end of \Cref{sec:main} for further details on the derivation.

\begin{cor}\label{cor:upperBounds}
~
    \begin{itemize}
        \item[(a)] For $r>k \ge 2$, we have \[s(\Ex^r(K^k_{k+1}))\le \frac{1}{r-k}.\]
        \item[(b)] For $\ell \ge 1$ and $r\ge 3$, we have
        \[s(C^r_{2\ell+1})\le\frac{2\ell-1}{r-2}.\]
        \item[(c)] For $r\ge k\ge 2$ and any tight $k$-tree $T$, $s(\Ex^r(T))=0$. That is, expansions of tight trees are Sidorenko. 
        \item[(d)] For $\ell \ge 1$ and $r\ge 3$ we have $s(C_{2\ell}^r)=0$.  That is, loose even cycles are Sidorenko.
    \end{itemize}
\end{cor}

In addition to these immediate consequences of \Cref{thm:main}, we prove two new results related to $s(F)$ and expansions.  First, we show that expansions of $F$ which contain  $K_{k+1}^k$ are not Sidorenko.

\begin{thm}\label{thm:containKk}
    If $F$ is a $k$-graph which contains $K_{k+1}^k$ as a subgraph, then for all $r>k$ we have 
    \[s(\Ex^r(F))\ge\frac{1}{r-k}.\]
    In particular, $\Ex^r(F)$ is not Sidorenko.
\end{thm}

The proof of Theorem~\ref{thm:containKk} makes use of a construction of Gowers and Janzer~\cite{gowers2021generalizations}, which generalized the seminal construction of Ruzsa and Szemer\'edi\cite{ruzsa1978}. We note that the quantitative lower bound of \Cref{thm:containKk} combined with \Cref{thm:main} gives optimal lower bounds for the random Tur\'an number $\ex(G_{n,p}^r,\Ex^r(K_{k+1}^k))$; see \cite{nie2023random,nie2021triangle}.  Moreover, this result together with \Cref{cor:upperBounds}(a) gives the following.
\begin{cor}\label{cor:Simplex}
    For $r>k\ge 2$, we have
    \[s(\Ex^r(K_{k+1}^k))=\frac{1}{r-k}.\]
\end{cor}
Corollary~\ref{cor:Simplex}, in other words, says that the minimum $t_{\Ex^r(K_{k+1}^k)}(H)$ among all $r$-graph $H$ with edge density $\delta$ is $\delta^{k+1+\frac{1}{r-k}\pm o(1)}$. For $\Ex^3(K_{3})=C_3^3$ this result was proved previously by Fox, Sah, Sawhney, Stoner, and Zhao~\cite[Theorem 1.2]{fox2020triforce}.  We note that the $r=k+1$ case of this corollary gives an example where the general upper bound of \Cref{cor:General} is tight. 

Finally, we establish an upper bound on $s(\Ex^r(F))$ in terms of $s(F)$.

\begin{thm}\label{thm:expansions}
    If $F$ is a $k$-graph with $s(F)<\infty$, then for all $r\ge k$ we have
    \[s(\Ex^r(F))\le \frac{v(F)-k}{v(F)-k+(r-k)(s(F)+e(F)-1)}\cdot s(F). \]
\end{thm}
For example, one can check that knowing $s(\Ex^{k+1}(K_{k+1}^k))\le 1$ from \Cref{cor:upperBounds} together with \Cref{thm:expansions} implies $s(\Ex^{r}(K_{k+1}^k))\le \frac{1}{r-k}$ for all $r>k$, and similarly one recovers our upper bound for $C_{2\ell+1}^r$ assuming the upper bound for $r=3$.   We also obtain the following nice corollary by taking $s(F)=0$.
\begin{cor}\label{cor:Sidorenko}
    If $F$ is a Sidorenko $k$-graph, then its expansions $\Ex^r(F)$ are Sidorenko for all $r\ge k$.
\end{cor}

\section{Proof of \Cref{thm:main} and its Corollaries}\label{sec:main}
Our proof of \Cref{thm:main} is based off the proof of \Cref{thm:CLS} from \cite{conlon2023extremal} which relies on the tensor product trick.  Given two $r$-graphs $G,H$, we define the \textit{tensor product} $G\otimes H$ to be $r$-graph on $V(G)\times V(H)$ where $((x_1,y_1),\ldots,(x_r,y_r))\in E(G\otimes H)$ if and only if $(x_1,\ldots,x_r)\in E(G)$ and $(y_1,\ldots,y_r)\in E(H)$.  For $N$ a positive integer, we define the $N$-fold tensor product $H^{\otimes N}$ inductively by setting $H^{\otimes 1}=H$ and $H^{\otimes N}=H\otimes H^{\otimes(N-1)}$.  The key property we need regarding tensor products is the fact that for any $r$-graphs $F,H$ and $N\ge 1$, we have
\[t_F(H^{\otimes N})=t_F(H)^N,\]
which is straightforward to verify.

By incorporating the tensor product trick from \cite{conlon2023extremal} together with random homomorphisms, we can show that $r$-graphs $G$ with few copies of $F$ have large $F$-free subgraphs; by copies of $F$ we mean subgraphs of $G$ that are isomorphic to $F$. To this end, we let $\c{N}_F(G)$ denote the number of copies of $F$ in $G$ and recall that $\ex(G,F)$ is the maximum number of edges in an $F$-free subgraph of $G$.

\begin{lem}\label{lem:relative}
    If $F$ is an $r$-graph such that there exists an $r$-graph $H$ with $t_{K_r^r}(H)=\al$ and $t_{F}(H)=\al^{s+e(F)}$, then for all $r$-graphs $G$ and integers $N\ge 1$ we have
    \[\ex(G,F)\ge \al^N e(G)-\al^{(s+e(F))N} \c{N}_F(G).\]
\end{lem}
\begin{proof}
    Let $\phi:V(G)\to V(H^{\otimes N})$ be chosen uniformly at random, and let $G'\sub G$ be the subgraph consisting of all hyperedges $e\in G$ which are mapped bijectively onto an edge of $H^{\otimes N}$. By linearity of expectation, we have \[\E[e(G')]=t_{K_r^r}(H^{\otimes N}) \cdot e(G)=\al^N \cdot e(G),\] and \[\E[\c{N}_F(G')]=t_{F}(H^{\otimes N})\cdot \c{N}_F(G)= \al^{(s+e(F))N} \cdot \c{N}_F(G).\]  Thus if we define $G''\sub G'$ by deleting an edge from each copy of $F$ in $G'$, then $G''$ is $F$-free and satisfies
    \[\E[e(G'')]\ge \al^N e(G)-\al^{(s+e(F))N} \c{N}_F(G),\]
    and hence there must exist an $F$-free subgraph of $G$ with at least this many edges, proving the result.
\end{proof}
With this we can prove our main result.
\begin{proof}[Proof of \Cref{thm:main}]
    Recall that we wish to prove if $F$ is an $r$-graph with $e(F)\ge 2$ and $\frac{v(F)-r}{e(F)-1}<r$, then for any $p=p(n)\ge n^{-\frac{v(F)-r}{e(F)-1}}$, we have a.a.s.
    \[\ex(G_{n,p}^r,F)\ge n^{r-\frac{v(F)-r}{e(F)-1}-o(1)}(p n^{\frac{v(F)-r}{e(F)-1}})^{\frac{s(F)}{e(F)-1+s(F)}}.\]
    If $s(F)=0$ then this result follows by a standard random deletion argument, so from now on we assume $s(F)>0$.
    
    Consider any $0<\ep\le s(F)$ (which exists by assumption $s(F)>0$).  By definition of $s(F)$, there exists a non-empty $r$-graph $H$ with $t_{K_r^r}(H)=\al>0$ and $t_{F}(H)=\al^{s+e(F)}$ with $0\le s(F)-\ep\le s\le s(F)$. By \Cref{lem:relative} we find
    \begin{equation}\ex(G_{n,p}^r,F)\ge \al^N e(G_{n,p}^r)-\al^{(s+e(F))N} \c{N}_F(G_{n,p}^r),\label{eq:relative}\end{equation}
    so it suffices to choose an $N$ such that this is sufficiently large a.a.s.

    Given $p$ and any function $\del(n)=o(1)$, let $N\ge 1$ be the smallest integer such that
    \[q:=\del(n)n^{-\frac{v(F)-r}{e(F)-1+s}}p^{-\frac{e(F)-1}{e(F)-1+s}}\ge \al^N.\]
    Note that such an integer exists since $0<\al<1$. Also note that $\al^N\le q\le \al^{N-1}$ by the minimality of $N$.
    
    Let $A$ denote the event that $e(G_{n,p}^r)\ge \frac{1}{2r!} pn^r$.  Because $e(G_{n,p}^r)$ is a binomial random variable and $pn^r\ge n^{r-\frac{v(F)-r}{e(F)-1}}\to \infty$ by hypothesis, the Chernoff bound implies that the event $A$ holds a.a.s.
    
    Let $B$ denote the event that $\c{N}_F(G_{n,p}^r)\le \del(n)^{-1/2} p^{e(F)} n^{v(F)}$. Since $\E[\c{N}_F(G_{n,p}^r)]\le p^{e(F)} n^{v(F)}$, it follows by Markov's inequality and $\del(n)=o(1)$ that $B$ holds a.a.s.
    
    Because $A\cap B$ hold a.a.s., we find that a.a.s.\ the bound in \eqref{eq:relative} is at least
    \begin{align*}\frac{1}{2r!} \al^N  pn^r-\del(n)^{-1/2}\al^{(s+e(F))N} p^{v(F)}n^{e(F)}&\ge \frac{1}{2r!} \al^N  pn^r(1-2r!\del(n)^{-1/2}q^{s+e(F)-1} p^{e(F)-1}n^{v(F)-r})\\&=\frac{1}{2r!} \al^N  pn^r(1-2r!\del(n)^{s+e(F)-3/2}).\end{align*}
    Note that $s+e(F)-3/2>0$ since $s\ge 0$ and $e(F)\ge 2$.  Thus for $n$ sufficiently large the quantity above is at least 
    \begin{align*}\frac{1}{4r!} \al^N  pn^r&\ge \frac{\al}{4r!} \del(n) n^{r-\frac{v(F)-r}{e(F)-1}}(p n^{\frac{v(F)-r}{e(F)-1}})^{\frac{s}{e(F)-1+s}}\\ &\ge \frac{\al}{4r!} \del(n) n^{-\frac{\ep}{e(F)-1+s(F)}}\cdot n^{r-\frac{v(F)-r}{e(F)-1}}(p n^{\frac{v(F)-r}{e(F)-1}})^{\frac{s(F)}{e(F)-1+s(F)}},\end{align*}
    with this last step used $s\ge s(F)-\ep$.  As $\ep>0$ was arbitrary and $\del(n)$ tends to 0 arbitrarily slowly, we conclude the desired result.
\end{proof}
As an aside, the bound of \Cref{thm:main} continues to hold in expectation even if $\frac{v(F)-r}{e(F)-1}\ge r$.  However, in this case we can not say $G_{n,p}^r$ has any edges a.a.s., and hence no non-trivial lower bound for $\ex(G_{n,p}^r,F)$ can hold a.a.s.

Focusing on the $p=1$ case, \Cref{lem:relative} quickly gives the following.
\begin{cor}\label{cor:classicTuran}
    If $F$ is an $r$-graph such that there exists a non-empty $r$-graph $H$ with $t_F(H)=t_{K_r^r}(H)^{s+e(F)}$, then
    \[\ex(n,F)=\Om\left(n^{r-\frac{v(F)-r}{e(F)-1}+\frac{(v(F)-r)s}{(e(F)-1)(s+e(F)-1)}}\right).\]
\end{cor}
\begin{proof}
    Take $G=K_n^r$ in \Cref{lem:relative}, which means $e(G)\ge \frac{1}{2 r!} n^r$ and $\c{N}_F(T)\le n^{v(F)}$, so for $\al=t_{K_r^r}(H)$ and any $N\ge 1$ we have
    \[\ex(n,F)=\ex(G,F)\ge \frac{1}{2r!}\al^N n^r\left(1-2r! \al^{(s+e(F)-1)N}n^{v(F)-r}\right).\]
    We conclude the result by taking $N$ such that $\al^N$ is a sufficiently small constant times $n^{-\frac{v(F)-r}{s+e(F)-1}}=n^{(v(F)-r)\left(\frac{-1}{e(F)-1}+\frac{s}{(e(F)-1)(s+e(F)-1)}\right)}$.
\end{proof}
As a point of comparison, a more careful analysis of the proof giving \Cref{thm:CLS} yields the following quantitative bound.
\begin{thm}[Quantitative \Cref{thm:CLS}]\label{thm:CLSquant}
    If $F$ is an $r$-graph such that there exists a non-empty $r$-graph $H$ with $t_F(H)=t_{K_r^r}(H)^{s+e(F)}$ and $t_{K_r^r}(H)=v(H)^{-\del}$, then
    \[\ex(n,F)=\Om\left(n^{r-\frac{v(F)-r}{e(F)-1}+\frac{\del s}{e(F)-1}}\right).\]
\end{thm}
It is not difficult to show  that $\del\le  \frac{v(F)-r}{s+e(F)-1}$ in \Cref{thm:CLSquant} (see \Cref{lem:trivBound} for a formal proof).  If $\del$ obtains this maximum possible value then \Cref{thm:CLSquant} matches \Cref{cor:classicTuran}; otherwise \Cref{cor:classicTuran} does strictly better.

We now sketch how \Cref{thm:main} together with known results for random Tur\'an bounds implies Corollaries \ref{cor:General} and \ref{cor:upperBounds}.

\begin{proof}[Proof of \Cref{cor:General}]
    Recall that we wish to show that if $F$ is an $r$-graph with $\ex(n,F)=O(n^\al)$, then
    \[s(F)\le \frac{v(F)-\al}{r-\al}-e(F).\]
    Let $s$ be such that there exists a non-empty $r$-graph $H$ with $t_F(H)=t_{K_r^r}(H)^{s+e(F)}$.  By \Cref{cor:classicTuran}, we have
    \[\Om(n^{r-\frac{v(F)-r}{s+e(F)-1}})=\ex(n,F)=O(n^\al).\]
    This implies $r-\frac{v(F)-r}{s+e(F)-1}\le \al$, and rearranging gives $s\le \frac{v(F)-\al}{r-\al}-e(F)$.  As $s(F)$ is the supremum over all such $s$, we conclude the result.
\end{proof}

\begin{proof}[Proof of \Cref{cor:upperBounds}]
    Throughout we implicitly utilize the fact that every $F$ we consider is $r$-balanced and hence $m_r(F)=\frac{v(F)-r}{e(F)-1}$.
    
    We first show (a): for $r>k\ge 2$ that $s(\Ex^r(K_{k+1}^k))\le \frac{1}{r-k}$.  By \cite[Theorem 1.4]{nie2023random}, we have
    for $p= n^{-\frac{1}{m_r(\Ex^r(K_{k+1}^k))}+c}=n^{-r+k-\frac{1}{k}+c}$ with sufficiently small $c=c(k,r)>0$ that a.a.s.
    $$
    \ex(G^r_{n,p}, \Ex^r(K^k_{k+1}))= p^{\frac{1}{(r-k)k+1}}n^{k+o(1)},
    $$
    hence by Theorem~\ref{thm:main}
    $$
    \frac{s(\Ex^r(K^k_{k+1}))}{k+s(\Ex^r(K^k_{k+1}))}\le \frac{1}{(r-k)k+1},
    $$
    which implies the desired upper bound.

    For (b), by \cite[Theorem 1.8]{nie2023random}, we have 
    for $p=n^{-\frac{1}{m_r(C_{2\ell+1}^r)}+c}=n^{-r+1+\frac{1}{2\ell}+c}$ with sufficiently small $c=c(\ell,r)>0$ that a.a.s.
    $$
    \ex(G^r_{n,p}, C^r_{2\ell+1})\le p^{\frac{2\ell-1}{2\ell(r-1)-1}}n^{2+o(1)},4
    $$
    hence by Theorem~\ref{thm:main}
    $$
    \frac{s(C^r_{2\ell+1})}{2\ell+s(C^r_{2\ell+1})}\le \frac{2\ell-1}{2\ell(r-1)-1},
    $$
    which implies
    $$
    s(C^r_{2\ell+1})\le \frac{2\ell-1}{r-2}.
    $$

    For (c), by \cite[Theorem 1.5]{nie2023random}, we have $\ex(G_{n,p}^r,\Ex^r(T))=n^{k-1+o(1)}$ a.a.s.\ for $p=n^{-\frac{1}{m_r(\Ex^r(T))}+c}=n^{-r+k-1+c}$ with sufficiently small $c=c(k,r)>0$, which implies the bound.
    
    For (d), it was proven in \cite{mubayi2023random,nie2024turan} that $\ex(G_{n,p}^r,C_{2\ell}^r)=n^{1+\frac{1}{2\ell-1}+o(1)}$ a.a.s.\ for $p=n^{-\frac{1}{m_r(C_{2\ell}^r)}+c}=n^{-r+1+\frac{1}{2\ell-1}+c}$ with $c=c(\ell,r)>0$, from which the bound follows.
\end{proof}

\section{Proof of \Cref{thm:containKk}}
To prove our lower bounds on $s(\Ex^r(F))$ when $F$ contains $K_{k+1}^k$, we need the following construction of Gowers and Janzer~\cite{gowers2021generalizations}, which is a generalization of the seminal construction of Ruzsa and Szemer\'edi~\cite{ruzsa1978}.

\begin{thm}[\cite{gowers2021generalizations,ruzsa1978}]\label{theorem:GJ}
For $r>k\ge 2$ and $n\ge 1$, there exists a graph $G_{n,r,k}$ on $n$ vertices with the following two properties:
\begin{itemize}
    \item[(i)] It has $n^ke^{-O(\sqrt{\log n})}$ subgraphs isomorphic to $K_r$;
    \item[(ii)] For any $t$ with $k<t\le r$ and any subgraph $G_1$ isomorphic to $K_k$, if there exist a subgraph $G_2$ isomorphic to $K_t$ and a subgraph $G_3$ isomorphic to $K_r$ such that $G_1\subseteq G_2$ and $G_1\subseteq G_3$, then $G_2\subseteq G_3$.
\end{itemize}
\end{thm}
Property $(ii)$ is indeed slightly stronger than the original statement of Theorem 1.2 in~\cite{gowers2021generalizations} which states ``every $K_k$ is contained in at most one $K_r$". This strengthened property is observed by Nie in~\cite{nie2023random}. In fact, property $(ii)$ is inherently implied by the proof of Lemma 3.1 in~\cite{gowers2021generalizations}. To see this, we roughly explain the idea of constructing $G_{n,r,k}$:
\begin{itemize}
    \item[1.] Let $S_{k-1}\subseteq \R^k$ be the $(k-1)$-dimension unit sphere. Pick $r$ points $p_1,\dots, p_r\in S_{k-1}$ in general position.
    \item[2.] Randomly pick $r$ point sets $V_1,\dots, V_r\subseteq S_{d-1}\subseteq \R^d$, each of size $n/r$, where $d$ is some suitable integer depending on $n$. By deleting a small portion of clustering points, we can make sure that all these points are reasonably well separeated. These points form the vertex set of $G_{n,r,k}$
    \item[3.] For any $1\le i<j\le r$ and points $v_i\in V_i, v_j\in V_j$, they form an edge in $G_{n,r,k}$ if and only if $|\langle v_i,v_j\rangle-\langle p_i,p_j\rangle|<\epsilon$, where $\epsilon$ is some resonably small constant.
\end{itemize}
By design, any copy of $K_r$ in $G_{n,r,k}$ must be close to a configuration with angles determined by the points $p_1,\dots, p_r$. Note that once we fix $k$ points $v_1\in V_1,\dots, v_k\in V_k$, then for all $k+1\le i\le r$, $v_i$ is constrained to lie in a set with small diameter. Since all points are well separated, there are at most one choice for each $v_i$. This property guarantees property \emph{(ii)} in \Cref{theorem:GJ}—See~\cite {gowers2021generalizations} for detailed computations.

Consider an $r$-graph $H_{n,r,k}$ on $V(G_{n,r,k})$ whose edges are the vertex sets of copies of $K_r$ in $G_{n,r,k}$. The following properties of $H_{n,r,k}$ is proved in~\cite{nie2023random}. We include a proof for completeness.
\begin{prop}[Proposition 5.4,~\cite{nie2023random}]\label{proposition:GJ1}
For $r>k\ge 2$ and $n\ge1$, $H_{n,r,k}$ has the following properties:
\begin{itemize}
    \item[(i)] $e(H_{n,r,k})\ge n^ke^{(-O(\sqrt{\log n}))}$;
    \item[(ii)] Any two edges intersect in at most $k-1$ vertices;
    \item[(iii)] $H_{n,r,k}$ does not contain any subgraph isomorphic to $\Ex^r(K_{k+1}^{k})$.
\end{itemize}
\end{prop}
\begin{proof}
Property (i) follows immediately. Now assume, for contradiction, that there exist two edges in the hypergraph that share at least $k$ vertices. This would imply the presence of a $K_k$ subgraph lying in the intersection of two distinct $K_r$-subgraphs within $G_{n,r,k}$, violating condition (ii) of $G_{n,r,k}$. 

Likewise, suppose that $H_{n,r,k}$ contains a subhypergraph isomorphic to $\Ex^r(K_{k+1}^k)$. Then, in the underlying graph $G_{n,r,k}$, this would correspond to a $K_k$ subgraph that is a common part of a $K_{k+1}$ and a $K_r$ which does not include the $K_{k+1}$. This again contradicts property (ii) of $G_{n,r,k}$.
\end{proof}

Now we are ready to prove our main result for this section. 
\begin{proof}[Proof of \Cref{thm:containKk}]
Recall that we wish to prove that if $F$ is a $k$-graph which contains $K_{k+1}^k$ as a subgraph, then for all $r>k$ we have
\[s(\Ex^r(F))\ge \frac{1}{r-k}.\]
That is, for any $\ep>0$ we want to find an $r$-graph $H$ such that
\[t_{\Ex^r(F)}(H)\le t_{K_r^r}(H)^{e(F)+\frac{1}{r-k}-\ep}.\]

Let $H=H_{n,r,k}$ with $n$ to be chosen sufficiently large in terms of $\ep$.  The crucial observation is the following. 
\begin{claim}\label{cl:GowJan}
    If $\phi:V(\Ex^r(K_{k+1}^k))\to V(H)$ is a homomorphism, then $\phi(e)=\phi(f)$ for all $e,f\in\Ex^r(K_{k+1}^k)$. 
\end{claim}

\begin{proof}
Denote the edges of $\Ex^r(K^k_{k+1})$ by $e_1, \dots, e_{k+1}$, where each edge $e_i$ is defined as 
\[ e_i = \{v_1, \dots, v_{k+1}, w_{i,1}, \dots, w_{i, r-k}\} \setminus \{v_i\}. \]
Assume, without loss of generality, that $\phi(e_1) \ne \phi(e_2)$ for the sake of contradiction. By Proposition~\ref{proposition:GJ1}(ii), it follows that
\[ |\phi(e_1) \cap \phi(e_2)| \le k-1. \]
However, observe that the image of the intersection satisfies
\[ \phi(e_1 \cap e_2) \subseteq \phi(e_1) \cap \phi(e_2), \]
and since $|e_1 \cap e_2| = k-1$, the size of the left-hand side is $k-1$, forcing equality:
\[ \phi(e_1 \cap e_2) = \phi(e_1) \cap \phi(e_2). \]
In particular, this implies that $\phi(v_1) \notin \phi(e_1)$.

Observe that $v_1$ belongs to every edge $e_i$ with $i > 1$, and thus $\phi(v_1) \in \phi(e_i)$ for each $i > 1$. Consequently, since $\phi(v_1) \notin \phi(e_1)$ as established earlier, it follows that $\phi(e_1) \ne \phi(e_i)$ for all $i > 1$. Therefore, we have
\[
|\phi(e_1) \cap \phi(e_i)| = k - 1 \quad \text{for all } i > 1.
\]
By the symmetry of the construction, it follows that
\[
|\phi(e_i) \cap \phi(e_j)| = k - 1 \quad \text{for all } i \ne j.
\]
Hence, the set of images $\phi(e_1), \dots, \phi(e_{k+1})$ forms a subhypergraph in $H_{m,r,k}$ isomorphic to $\Ex^r(K^k_{k+1})$, which contradicts Proposition~\ref{proposition:GJ1}(iii).
\end{proof}
This allows us to prove the following.

\begin{claim}
    If $x\in V(F)$ is contained in a $K_{k+1}^k$, then for any map $\phi:V(F)\sm \{x\}\to V(H)$, there are at most $O(1)$ homomorphisms $\phi':V(\Ex^r(F))\to V(H)$ such that the restriction $\phi'|_{V(F)\sm \{x\}}$ equals $\phi$.
\end{claim}
\begin{proof}
    Let $x_1,\ldots,x_k$ be the other vertices of the $K_{k+1}^k$ containing $x$.  Because $\Ex^r(F)$ contains an edge $e$ containing $\{x_1,\ldots,x_k\}$, if the set $X=\{\phi(x_1),\ldots,\phi(x_k)\}$ either has size less than $k$ or is not contained in an edge of $H$, then no homomorphism restricts to $\phi$, so we may assume this is not the case.  By \Cref{proposition:GJ1}(ii), there exists a unique edge $h\in H$ containing $X$, and any homomorphism $\phi'$ which restricts to $\phi$ must map $e$ to $h$.  By \Cref{cl:GowJan}, we must have $\phi'(x)\in h$.

    Let $y\in h$ and define $\phi_y:V(F)\to V(H)$ by having $\phi_y(x)=y$ and $\phi_y(z)=\phi(z)$ for all other $z$.  By the observation above, any $\phi'$ which restricts to $\phi$ must restrict to $\phi_y$ for one of the at most $O(1)$ choices $y\in h$.  We claim that for any $y\in h$ there are at most $O(1)$ homomorphism $\phi'$ which restricts to $\phi_y$, from which the result will follow.

    Indeed, consider any vertex $z\in V(\Ex^r(F))\sm V(F)$, which by definition of the expansion means there is an edge $\{z_1,\ldots,z_k\}\in E(F)$ such that $\{z,z_1,\ldots,z_k\}$ is contained in an edge $e'$ of $\Ex^r(F)$.  If the set $Z=\{\phi_y(z_1),\ldots,\phi_y(z_k)\}$ has size less than $k$ or is not contained in an edge of $H$, then no homomorphism restricts to $\phi_y$, so we may assume this is not the case.  By \Cref{proposition:GJ1}(ii), there exists a unique edge $h'\in H$ containing $Z$, and any homomorphism $\phi'$ which restricts to $\phi_y$ must map $e'$ to $h'$.  In conclusion, for any $z\in V(\Ex^r(F))\sm V(F)$ there are at most $O(1)$ vertices $z$ can map to in a homomorphism $\phi'$ which restricts to $\phi_y$.  Thus there are at most $O(1)$ homomorphisms $\phi'$ which restrict to $\phi_y$, proving the claim.
\end{proof}
Observe that the number of maps $\phi:V(F)\sm \{x\}\to V(H)$ is at most $n^{v(F)-1}$, and hence the claim above implies
\[t_{\Ex^r(F)}(H)\le n^{v(F)-1-v(\Ex^r(F))}=n^{-1-(r-k)e(F)}.\]
On the other hand, 
\[t_{K_r^r}(H)^{e(F)+\frac{1}{r-k}-\ep}=n^{(k-r-o(1))(e(F)+\frac{1}{r-k}-\ep)}=n^{-1-(r-k)e(F)+(r-k)\ep-o(1)},\]
and for $n$ sufficiently large in terms of $\ep$ this is greater than the bound for $t_{\Ex^r(F)}(H)$ obtained above, proving the result.

\end{proof}
Before going on, we note that one can easily adapt the proof above to give stronger quantitative bounds on $s(F)$ in certain cases.  For example, if there exists a subset $V\sub V(F)$ such that for every $x\in V$ there exist vertices $x_1,\ldots,x_k\in V(F)\sm V$ forming a $K_{k+1}^k$ with $x$, then one can prove
\[s(\Ex^r(F))\ge \frac{|V|}{r-k}.\]

\section{Proof of \Cref{thm:expansions}}
Here we establish an upper bound on $s(\Ex^r(F))$ in terms of $F$.  For this the following will be useful.
\begin{lem}\label{lem:trivBound}
    If $F'$ is an $r$-partite $r$-graph and $H$ is an $r$-graph such that $t_{F'}(H)\le t_{K_r^r}(H)^{s'+e(F')}$ for some $s'$, then $t_{K_r^r}(H)\ge v(H)^{-\frac{v(F')-r}{s'+e(F')-1}}$.
\end{lem}
\begin{proof}
    Let $H$ satisfy $t_{F'}(H)\le  t_{K_r^r}(H)^{s'+e(F')}$ and let $\del$ be such that $t_{K_r^r}(H)=v(H)^{-\del}$.  This means $H$ has $(r!)^{-1} v(H)^{r-\del}$ edges, and hence $\hom(F',H)\ge v(H)^{r-\del}$ (since $F'$ mapping onto a single edge of $H$ is always a homomrphism by assumption of $F'$ being $r$-partite).  Hence \[v(H)^{r-\del-v(F')}\le t_{F'}(H)\le v(H)^{-\del(s'+e(F'))},\]
    and rearranging shows $\del\le \frac{v(F')-r}{s'+e(F')-1}$, proving the result.
\end{proof}
For our proof, it will be convenient to work with weighted $r$-graphs $W$ where the weight of an $r$-set $\{x_1,\ldots,x_r\}$ is denoted $W(x_1,\ldots,x_r)$.  We let $\Hom(F,W)$ denote the set of all maps $\phi:V(F)\to V(W)$ which are injective on $e\in E(F)$.  We define the weight $w(\phi)$ of $\phi\in \Hom(F,W)$ to be $\prod_{e\in E(F)} W(\phi(e))$ and we define $\hom(F,W)=\sum_{\phi\in \Hom(F,W)}w(\phi)$.  With this we can define the notion of homomorphism densities $t_F(W)$ exactly as before, and it is not difficult to show that if $s(F)=s$ then $t_F(W)\ge t_{K_r^r}(W)^{s+e(F)}$ for all weighted $r$-graphs $W$.

\begin{proof}[Proof of \Cref{thm:expansions}]
Recall that we wish to show that if $F$ is a $k$-graph with $s(F)<\infty$, then for all $r\ge k$ we have
 \[s(\Ex^r(F))\le s':= \frac{v(F)-k}{v(F)-k+(r-k)(s(F)+e(F)-1)}\cdot s(F). \]
 Assume for contradiction that there exists an $n$-vertex $r$-graph $H$ such that $t_{\Ex^r(F)}(H)<t_{K_r^r}(H)^{s'+e(F)}$.
 Since $s(F)<\infty$ by assumption, $F$ must be $k$-partite and hence $\Ex^r(F)$ must be $r$-partite.  Thus by \Cref{lem:trivBound} we must have $t_{K_{r}^{r}}(H)\ge  n^{-\frac{v(F)+(r-k)e(F)-r}{s'+e(F)-1}}$, or equivalently
 \begin{equation}n\ge t_{K^r_r}(H)^{-\frac{s'+e(F)-1}{v(F)+(r-k)e(F)-r}}.\label{eq:vertBound}\end{equation}

 We define an auxiliary weighted $k$-graph $W$ on $V(H)$ such that for any $k$-set $X$ we have $W(X)=(r-k)\,!\deg_H(X)$, i.e.\ $(r-k)!$ times the number of edges of $H$ containing $X$.  By definition of $s(F)$ we have

\begin{equation}\label{eq:tF(W)_lowerbound}
\begin{aligned}
&t_{F}(W)=\frac{\sum_{\phi\in \Hom(F,W)}\prod_{e\in E(F)}(r-k)\,!\deg_H(\phi(e))}{n^{v(F)}}\\
&\ge t_{K^k_k}(W)^{s(F)+e(F)}=\left(\frac{k\,!}{n^k}\sum_{X\in W}(r-k)\,!\deg_H(X)\right)^{s(F)+e(F)}=\left(\frac{r\,!e(H)}{n^k}\right)^{s(F)+e(F)}.
\end{aligned}    
\end{equation}

By definition of expansions, every homomorphism $\phi:V(\Ex^r(F))\to V(H)$ can be formed by first choosing a homomorphism $\phi':V(F)\to V(W)$, and then for each $e'\in F$ with $e\in \Ex^r(F)$ the edge containing $e'$, one chooses some edge $h\in E(H)$ containing the $k$-set $\phi(e')$ together with a bijection from $e\setminus e' $ to $h\setminus \phi(e)$.  Thus we have 
$$
\hom(\Ex^r(F),H)=\sum_{\phi\in \Hom(F,W)}\prod_{e\in E(F)}(r-k)\,!\deg_H(\phi(e)).
$$
Hence,
$$
\begin{aligned}
t_{\Ex^r(F)}(H)&=\frac{\sum_{\phi\in \Hom(F,W)}\prod_{e\in E(F)}(r-k)\,!\deg_H(\phi(e))}{n^{v(F)+(r-k)e(F)}}\\
&=\frac{t_F(W)}{n^{(r-k)e(F)}}\\
&\ge \frac{\l(r\,!e(H) n^{-k}\r)^{s(F)+e(F)}}{n^{(r-k)e(F)}}\\
&=t_{K_r^r}(H)^{s(F)+e(F)}\cdot n^{(r-k)s(F)}\\
&\ge t_{K_r^r}(H)^{s(F)+e(F)}\cdot t_{K_r^r}(H)^{-\frac{(r-k)s(F)(s'+e(F)-1)}{v(F)+(r-k)e(F)-r}},
\end{aligned}
$$
where the first inequality used \eqref{eq:tF(W)_lowerbound} and the second used \eqref{eq:vertBound}. One can verify that this final quantity equals $t_{K_r^r}(H)^{s'+e(F)}$, a contradiction to our choice of $H$.  We conclude the result.
\end{proof}

As an aside, it is tempting to generalize the statement of \Cref{thm:expansions} to hold even at $s(F)=\infty$.  Indeed, ``taking the limit'' in \Cref{thm:expansions} suggests that when $s(F)=\infty$ we should have
\[s(\Ex^r(F))\le \frac{v(F)-k}{r-k}.\]
This  does hold whenever $\Ex^r(F)$ satisfies $\ex(n,\Ex^r(F))=O(n^k)$ by our general upper bound \Cref{cor:General} since
\[\frac{v(\Ex^r(F))-k}{r-k}-e(\Ex^r(F))=\frac{v(F)-k}{r-k}.\]
However, such a result does not hold in general.  For example, it certainly fails if $\Ex^r(F)$ is not $r$-partite, such as when considering $\Ex^3(K_t)$ with  $t>3$.

\section{Concluding Remarks}
There are many questions left to explore regarding (non-)Sidorenko hypergraphs which we break into three broad categories.

\textbf{Sidorenko Expansions.}  In \Cref{cor:Sidorenko} we showed that expansions of Sidorenko hypergraphs are Sidorenko.  This motivates the following conjecture.
\begin{conj}\label{conj:weakSidorenko}
    For every bipartite graph $F$, there exists an $r\ge 2$ such that $\Ex^r(F)$ is Sidorenko.
\end{conj}
Note that Sidorenko's conjecture predicts this holds with $r=2$ for all $F$.  Moreover, \Cref{cor:Sidorenko} suggests it may be easier to prove \Cref{conj:weakSidorenko} for larger values of $r$ (since if it holds for some $r_0$, then it holds for all $r\ge r_0$).  As such, \Cref{conj:weakSidorenko} can be viewed as a (potentially) weaker version of Sidorenko's conjecture, and it would be particularly interesting if one could verify it for some $r\ge 2$ independent of $F$.

Another question asks whether the converse of \Cref{cor:Sidorenko} holds.

\begin{quest}\label{quest:expansionConverse}
    Is it true that $F$ is Sidorenko if and only if all of its expansions $\Ex^r(F)$ are Sidorenko?
\end{quest}
Note that if \Cref{quest:expansionConverse} has an affirmative answer, then \Cref{conj:weakSidorenko} would be equivalent to Sidorenko's conjecture.  The simplest case that we do not know how to answer is the following.

\begin{quest}\label{quest:expandSidorenko}
    If $F$ is a non-bipartite graph, are all of its expansions $\Ex^r(F)$ not Sidorenko?
\end{quest}
We note that \cite[Theorem 3.1]{conlon2023extremal} shows this holds if $F$ has odd girth, but beyond this we know nothing. 

\textbf{$k$-linear Hypergraphs}.  In \Cref{thm:containKk} we proved that expansions of $k$-graphs containing $K_{k+1}^k$ are not Sidorenko.  We conjecture that the following stronger result holds.
\begin{conj}\label{conj:linear}
    If $F$ is an $r$-graph such that $|e\cap f|<k$ for any distinct $e,f\in F$ and such that $F$ contains an expansion $\Ex^r(K_{k+1}^k)$ as a subgraph, then $F$ is not Sidorenko.
\end{conj}
The case $k=2$ was proven in \cite{conlon2023extremal}, but as they note, their construction does not seem to effectively generalize to higher uniformities.  We offer an alternative proof of the $k=2$ case in the appendix of the arXiv version of this paper to serve as another potential source of inspiration towards proving \Cref{conj:linear}.

\textbf{Bounds for $s(F)$}.   \Cref{thm:main} motivates the problem of determining $s(F)$ for non-Sidorenko hypergraphs, especially those for which the random Tur\'an number is unknown.  One outstanding case is that of loose odd cycles.
\begin{prob}
    Determine $s(C_{2\ell+1}^r)$.
\end{prob}
In \Cref{cor:upperBounds} we showed $s(C_{2\ell+1}^r)\le \frac{2\ell-1}{r-2}$, and by considering $H=K_r^r$ it is possible to prove that $s(C_{2\ell+1}^r)$ is at least roughly $r^{-2\ell-1}$.  We believe the upper bound is closer to the truth, but we do not think this is tight.  Our best guess (though we would not go so far as to make it a conjecture) is that \[s(C_{2\ell+1}^r)=\frac{1}{(r-1)\ell-1}.\]
Indeed, the lower bound $s(C_{2\ell+1}^r)\ge \frac{1}{(r-1)\ell-1}$ would follow if there existed $n$-vertex $r$-graphs of girth $2\ell+2$ with $n^{1+1/\ell-o(1)}$ edges, which is the densest such an $r$-graph can be \cite{collier2018linear}.  Such $r$-graphs are only known to exist when $\ell=1$ due to Ruzsa-Szemer\'edi type constructions, and it is difficult for us to imagine a construction that would give a better lower bound for $s(C_{2\ell+1}^r)$ than this.  We also note that the general upper bound $s(C_{2\ell+1}^r)\le \frac{1}{(r-1)\ell-1}$ would follow from the $r=3$ result by using \Cref{thm:expansions}.



\section*{Acknowledgements}

We thank the anonymous referees for their careful reading of the paper and useful suggestions.


%
%

\bibliographystyle{alphaurl}
\bibliography{refs}

\newcommand{\etalchar}[1]{$^{#1}$}
\begin{thebibliography}{CKLL18}

\bibitem[BBM09]{balogh2009erdHos}
J{\'o}zsef Balogh, Tom Bohman, and Dhruv Mubayi.
\newblock Erd{\H{o}}s--{Ko}--{Rado} in random hypergraphs.
\newblock {\em Combinatorics, Probability and Computing}, 18(5):629--646, 2009.

\bibitem[BKL23]{balogh2023sharp}
J{\'o}zsef Balogh, Robert~A Krueger, and Haoran Luo.
\newblock Sharp threshold for the {E}rd{\H{o}}s--{K}o--{R}ado theorem.
\newblock {\em Random Structures \& Algorithms}, 62(1):3--28, 2023.

\bibitem[CCGJ18]{collier2018linear}
Clayton Collier-Cartaino, Nathan Graber, and Tao Jiang.
\newblock Linear {T}ur\'an numbers of linear cycles and cycle-complete {R}amsey numbers.
\newblock {\em Combinatorics, Probability and Computing}, 27(3):358--386, 2018.

\bibitem[CFS10]{conlon2010approximate}
David Conlon, Jacob Fox, and Benny Sudakov.
\newblock An approximate version of {S}idorenko’s conjecture.
\newblock {\em Geometric and Functional Analysis}, 20:1354--1366, 2010.

\bibitem[CG16]{conlon2016combinatorial}
David Conlon and William~Timothy Gowers.
\newblock Combinatorial theorems in sparse random sets.
\newblock {\em Annals of Mathematics}, pages 367--454, 2016.

\bibitem[CKLL18]{conlon2018some}
David Conlon, Jeong~Han Kim, Choongbum Lee, and Joonkyung Lee.
\newblock Some advances on {S}idorenko's conjecture.
\newblock {\em Journal of the London Mathematical Society}, 98(3):593--608, 2018.

\bibitem[CL17]{conlon2017finite}
David Conlon and Joonkyung Lee.
\newblock Finite reflection groups and graph norms.
\newblock {\em Advances in Mathematics}, 315:130--165, 2017.

\bibitem[CL18]{conlon2018sidorenko}
David Conlon and Joonkyung Lee.
\newblock {S}idorenko's conjecture for blow-ups.
\newblock {\em arXiv preprint arXiv:1809.01259}, 2018.

\bibitem[CLS23]{conlon2023extremal}
David Conlon, Joonkyung Lee, and Alexander {S}idorenko.
\newblock Extremal numbers and {S}idorenko's conjecture.
\newblock {\em arXiv preprint arXiv:2307.04588}, 2023.

\bibitem[CR21]{coregliano2021biregularity}
Leonardo~N Coregliano and Alexander~A Razborov.
\newblock Biregularity in {S}idorenko's conjecture.
\newblock {\em arXiv preprint arXiv:2108.06599}, 2021.

\bibitem[F\"94]{F}
Zolt\'{a}n F\"{u}redi.
\newblock Random {R}amsey graphs for the four-cycle.
\newblock {\em Discrete Math.}, 126(1-3):407--410, 1994.
\newblock \href {https://doi.org/10.1016/0012-365X(94)90287-9} {\path{doi:10.1016/0012-365X(94)90287-9}}.

\bibitem[FNW24]{frankl2024matching}
Peter Frankl, Jiaxi Nie, and Jian Wang.
\newblock On the matching problem in random hypergraphs.
\newblock {\em arXiv preprint arXiv:2410.15585}, 2024.

\bibitem[FSS{\etalchar{+}}20]{fox2020triforce}
Jacob Fox, Ashwin Sah, Mehtaab Sawhney, David Stoner, and Yufei Zhao.
\newblock Triforce and corners.
\newblock In {\em Mathematical Proceedings of the Cambridge Philosophical Society}, volume 169, pages 209--223. Cambridge University Press, 2020.

\bibitem[FW17]{fox2017local}
Jacob Fox and Fan Wei.
\newblock On the local approach to {S}idorenko's conjecture.
\newblock {\em Electronic Notes in Discrete Mathematics}, 61:459--465, 2017.

\bibitem[GHO17]{gauy2017erdHos}
Marcelo~M Gauy, Hiep Han, and Igor~C Oliveira.
\newblock Erd{\H{o}}s--{Ko}--{Rado} for random hypergraphs: asymptotics and stability.
\newblock {\em Combinatorics, Probability and Computing}, 26(3):406--422, 2017.

\bibitem[GJ21]{gowers2021generalizations}
WT~Gowers and Barnab{\'a}s Janzer.
\newblock Generalizations of the {R}uzsa--{S}zemer{\'e}di and rainbow tur{\'a}n problems for cliques.
\newblock {\em Combinatorics, Probability and Computing}, 30(4):591--608, 2021.

\bibitem[Hat10]{hatami2010graph}
Hamed Hatami.
\newblock Graph norms and {S}idorenko’s conjecture.
\newblock {\em Israel Journal of Mathematics}, 175:125--150, 2010.

\bibitem[HK19a]{hamm2019erdHosI}
Arran Hamm and Jeff Kahn.
\newblock On {Erd{\H{o}}s}--{Ko}--{Rado} for random hypergraphs {I}.
\newblock {\em Combinatorics, Probability and Computing}, 28(6):881--916, 2019.

\bibitem[HK19b]{hamm2019erdHos}
Arran Hamm and Jeff Kahn.
\newblock On {Erd{\H{o}}s}--{Ko}--{Rado} for random hypergraphs {II}.
\newblock {\em Combinatorics, Probability and Computing}, 28(1):61--80, 2019.

\bibitem[KLL16]{kim2016two}
Jeong~Han Kim, Choongbum Lee, and Joonkyung Lee.
\newblock Two approaches to {S}idorenko’s conjecture.
\newblock {\em Transactions of the American Mathematical Society}, 368(7):5057--5074, 2016.

\bibitem[Lov11]{lovasz2011subgraph}
L{\'a}szl{\'o} Lov{\'a}sz.
\newblock Subgraph densities in signed graphons and the local {S}imonovits--{S}idorenko conjecture.
\newblock {\em the electronic journal of combinatorics}, pages P127--P127, 2011.

\bibitem[LS11]{li2011logarithimic}
JL~Li and Bal{\'a}zs Szegedy.
\newblock On the logarithimic calculus and {S}idorenko's conjecture.
\newblock {\em arXiv preprint arXiv:1107.1153}, 2011.

\bibitem[MS16]{morris2016number}
Robert Morris and David Saxton.
\newblock The number of ${C}_{2\ell}$-free graphs.
\newblock {\em Advances in Mathematics}, 298:534--580, 2016.

\bibitem[MS23]{mckinley2023random}
Gwen McKinley and Sam Spiro.
\newblock The random {T}ur\'an problem for theta graphs.
\newblock {\em arXiv preprint arXiv:2305.16550}, 2023.

\bibitem[MY23]{mubayi2023random}
Dhruv Mubayi and Liana Yepremyan.
\newblock On the random {T}ur\'an number of linear cycles.
\newblock {\em arXiv preprint arXiv:2304.15003}, 2023.

\bibitem[Nie23]{nie2023random}
Jiaxi Nie.
\newblock Random {T}ur\'an theorem for expansions of spanning subgraphs of tight trees.
\newblock {\em arXiv preprint arXiv:2305.04193}, 2023.

\bibitem[Nie24]{nie2024turan}
Jiaxi Nie.
\newblock Tur{\'a}n theorems for even cycles in random hypergraph.
\newblock {\em Journal of Combinatorial Theory, Series B}, 167:23--54, 2024.

\bibitem[NSV21]{nie2021triangle}
Jiaxi Nie, Sam Spiro, and Jacques Verstra{\"e}te.
\newblock Triangle-free subgraphs of hypergraphs.
\newblock {\em Graphs and Combinatorics}, 37:2555--2570, 2021.

\bibitem[RS78]{ruzsa1978}
I.~Z. Ruzsa and E.~Szemer\'edi.
\newblock Triple systems with no six points carrying three triangles.
\newblock {\em Combinatorics (Keszthely, 1976), Coll. Math. Soc. J. Bolyai}, 18:939--945, 1978.

\bibitem[Sch16]{schacht2016extremal}
Mathias Schacht.
\newblock Extremal results for random discrete structures.
\newblock {\em Annals of Mathematics}, pages 333--365, 2016.

\bibitem[Sid91]{Sidorenko1991Inequalities}
A.~F. Sidorenko.
\newblock Inequalities for functionals generated by bipartite graphs.
\newblock {\em Diskret. Mat.}, 3(3):50--65, 1991.
\newblock \href {https://doi.org/10.1515/dma.1992.2.5.489} {\path{doi:10.1515/dma.1992.2.5.489}}.

\bibitem[Sid93]{Sidorenko1993Acorrelation}
Alexander Sidorenko.
\newblock A correlation inequality for bipartite graphs.
\newblock {\em Graphs Combin.}, 9(2):201--204, 1993.
\newblock \href {https://doi.org/10.1007/BF02988307} {\path{doi:10.1007/BF02988307}}.

\bibitem[Spi24]{spiro2024random}
Sam Spiro.
\newblock Random polynomial graphs for random tur{\'a}n problems.
\newblock {\em Journal of Graph Theory}, 105(2):192--208, 2024.

\bibitem[Sze14]{szegedy2014information}
Balazs Szegedy.
\newblock An information theoretic approach to {S}idorenko's conjecture.
\newblock {\em arXiv preprint arXiv:1406.6738}, 2014.

\end{thebibliography}

\newpage

\section*{Appendix: Linear Hypergraphs with Loose Triangles}

Here we give an alternative proof to the following result from \cite{conlon2023extremal}.
\begin{thm}[\cite{conlon2023extremal}]\label{thm:looseLinear}
    If $F$ is linear and contains a loose triangle, then $F$ is not Sidorenko.
\end{thm}
As our primary aim is to demonstrate a new approach and not to prove anything novel, we will be somewhat loose with the details. Our proof will rely on the language of graphons; we refer the reader to e.g.\ \cite{conlon2023extremal} for a refresher on the relevant definitions.
\begin{proof}

Our proof relies on the following basic claims.
\begin{claim}\label{cl:expand}
    Let $W$ be a graphon and $p$ the constant graphon of density $p$.  For any $r$-graph $F$, we have
    \[t_F(p+(1-p)W)=\sum_{F'\sub F} p^{e(F)-e(F')}(1-p)^{e(F')} t_{F'}(W).\]
\end{claim}
We say that a hypergraph $F$ is a \textit{forest} if one can order its edges $e_1,\ldots,e_m$ such that $e_i$ intersects $\bigcup_{j<i} e_j$ in at most one vertex.

\begin{claim}\label{cl:regularLoose}
    There exists a graphon $W$ with $t_F(W)=t_{K_r^r}(W)^k$ whenever $F$ is a forest with $k$ edges and where $t_F(W)<t_{K_r^r}(W)^3$ for any other $F$ which is linear.
\end{claim}
\begin{proof}
    Take $W$ to be the graphon corresponding to $K_r^r$, i.e.,  $W(x_1,\ldots,x_r)=1$ if $|\{x_1,\ldots,x_r\}\cap [(i-1)/r,i/r]|=1$ for all $1\le i\le r$ and $W(x_1,\ldots,x_r)=0$ otherwise.  It is not difficult to verify that $t_F(W)=t_{K_r^r}(W)^k$ when $F$ is a forest (this follows because $W$ is a ``regular'' graphon) and that $t_{C_4^r}(W)<t_{K_r^r}(W)^3$ (because $C_4^r$ contains a non-spanning forest on 3 edges).  One can also check that $t_{C_3^r}(W)<t_{K_r^r}(W)^3$, namely because $W$ has the same homomorphism densities as $K_r^r$ and because
    \[\frac{\hom(C_3^r,K_r^r)}{r^{v(C_3^r)}}=\frac{r(r-1)(r-2)\cdot [(r-2)!]^3}{r^{3r-3}}<\left(\frac{r!}{r^r}\right)^3,\]
    since this inequality is equivalent to saying $ r(r-1)(r-2)\cdot r^3< (r(r-1))^3$.
    Now every linear $F$ either contains a forest on at least 4 edges or contains $C_3^r$ or $C_4^r$, from which the result follows.
\end{proof}

We now prove the result.  Let $F$ be a linear $r$-graph containing a loose triangle.  Let $W$ be the graphon from \Cref{cl:regularLoose}, and possibly by taking $N$-fold tensor powers, we can assume its edge density $q:=t_{K_r^r}(W)$ is sufficiently small and that there exists some $c<1$ sufficiently small with $t_{F'}(W)\le c q^3$ for all $F'\sub F$ which are either loose triangles or which contain at least 4 edges.  We will show that $F$ is not Sidorenko by considering the graphon $\half+\half W$, which by \Cref{cl:expand} is equivalent to showing
\begin{equation}2^{-e(F)}\sum_k \sum_{F'\sub F: e(F')=k} t_{F'}(W)=t_F\left(\half+\half W\right)<t_{K_r^r}\left(\half+\half W\right)^{e(F)}=2^{-e(F)}\sum_k {e(F)\choose k} q^k.\label{eq:looseTriangle}\end{equation}
Because $F$ is linear, every $F'\sub F$ with at most 3 edges is either a forest or a loose triangle.  By definition of $W$, in the former case we have $t_{F'}(W)=q^k$, and in the latter case we have $t_{F'}(W)\le c q^3$.  As $F$ contains at least one loose triangle, we have
\[(1-c) q^3+\sum_{k\le 3} \sum_{F'\sub F: e(F')=k} t_{F'}(W)\le \sum_{k\le 3} \sum_{F'\sub F: e(F')=k} q^k\le \sum_{k\le 3} {e(F)\choose k} q^k.\]
Note that every $F'\sub F$ with $e(F')>3$ has $t_{F'}(W)\le c q^3$ by assumption.  Thus by taking $c$ sufficiently small we have $\sum_{k>3} \sum_{F'\sub F: e(F')=k} t_{F'}(W)<(1-c)q^3$, which combined with the equation above implies \eqref{eq:looseTriangle} as desired.
\end{proof}
We note that a nearly identical proof goes through to show that $F$ is not Sidorenko whenever $F$ has odd girth (which was also proven in \cite{conlon2023extremal}).

\end{document}